\documentclass[12pt]{amsart}
\usepackage{graphicx}
\usepackage{amssymb}

\usepackage{mathptmx}      

\usepackage{latexsym}

\newtheorem{theorem}{Theorem}[section]

\newtheorem{corollary}[theorem]{Corollary}
\newtheorem{proposition}[theorem]{Proposition}

\newtheorem{question}[theorem]{Question}

\theoremstyle{definition}
\newtheorem{definition}[theorem]{Definition}
\newtheorem{example}[theorem]{Example}

\theoremstyle{remark}
\newtheorem{remark}[theorem]{Remark}

\newcommand{\C}{\mathbb{C}}
\newcommand{\Cn}{\mathbb{C}^n}
\newcommand{\B}{\mathbb{B}}
\newcommand{\Bn}{\mathbb{B}^n}
\newcommand{\U}{\mathbb{U}}
\newcommand{\N}{\mathbb{N}}
\newcommand{\ov}{\overline}

\newcommand{\na}{\mathcal{N}_A}

\begin{document}

\title[Approximation properties of univalent mappings]{Approximation properties of univalent mappings on the unit ball
in $\Cn$
}

\author{Hidetaka Hamada, Mihai Iancu, Gabriela Kohr and
Sebastian Schlei{\ss}inger
}


\address{
H. Hamada
\\
Faculty of Engineering, Kyushu Sangyo University, 3-1 Matsukadai 2-Chome,
Higashi-ku Fukuoka 813-8503, Japan}
\email{h.hamada@ip.kyusan-u.ac.jp}

\address{M. Iancu
\\
Faculty of Mathematics and Computer Science,
Babe\c{s}-Bolyai University, 1 M. Kog\u{a}l\-niceanu Str., 400084
Cluj-Napoca, Romania}
\email{miancu@math.ubbcluj.ro}

\address{G. Kohr
\\
Faculty of Mathematics and Computer Science,
Babe\c{s}-Bolyai University, 1 M. Kog\u{a}l\-niceanu Str., 400084
Cluj-Napoca, Romania}
\email{gkohr@math.ubbcluj.ro}

\address{S. Schlei{\ss}inger
\\
Department of Mathematics,
University of W\"urzburg, Am Hubland,
97074 W\"urzburg, Germany}
\email{sebastian.schleissinger@mathematik.uni-wuerzburg.de}

\date{}

\maketitle

\begin{abstract}
Let $n\geq 2$.
In this paper, we obtain
approximation properties of various families of normalized univalent mappings $f$
on the Euclidean unit ball $\mathbb{B}^n$ in $\mathbb{C}^n$
by automorphisms of $\mathbb{C}^n$
whose restrictions to $\mathbb{B}^n$ have the same geometric property of $f$.
First, we obtain approximation properties
of spirallike, convex and $g$-starlike mappings $f$ on $\mathbb{B}^n$
by automorphisms of $\mathbb{C}^n$
whose restrictions to $\mathbb{B}^n$ have the same geometric property of $f$, respectively.
Next, for a nonresonant operator $A$ with $m(A)>0$,
we obtain an approximation property of
mappings which have $A$-parametric representation
by automorphisms of $\mathbb{C}^n$
whose restrictions to $\mathbb{B}^n$ have $A$-parametric representation.
Certain questions will be also mentioned.
Finally, we obtain an approximation property by automorphisms of $\mathbb{C}^n$
for a subset
of $S_{I_n}^0(\mathbb{B}^n)$ consisting of mappings $f$ which satisfy the condition
$\|Df(z)-I_n\|<1$, $z\in\mathbb{B}^n$. Related results will be also obtained.
\end{abstract}



\section{Introduction}
\label{intro}
\setcounter{equation}{0}
In this paper we continue the work related to embedding univalent mappings
in Loewner chains on the unit ball $\Bn$ in $\Cn$ (see \cite{Ar},
\cite{GHKK-JAM}, \cite{GHKK13}, \cite{GHKK15}, \cite{HIK}, \cite{Vo}).
If $f$ is a biholomorphic mapping on $\B^n$ such that $f(\B^n)$ is Runge, then
$f$ can be approximated locally uniformly on $\Bn$ by automorphisms of $\Cn$
by \cite[Theorem 2.1]{And-Lemp}, for all $n\geq 2$.
Then a natural question arises as follows:
Is it possible to approximate $f$ by automorphisms of $\Cn$
whose restrictions to $\Bn$ have the same geometric property of $f$?

In this paper, we obtain
approximation properties of various families of normalized univalent mappings $f$
on $\Bn$ by automorphisms of $\Cn$
whose restrictions to $\Bn$ have the same geometric property of $f$, for all $n\geq 2$.
Indeed, in view of a recent result of Hamada \cite{Ha2}, which yields that
any spirallike domain in $\Cn$
with respect to an operator $A\in L(\C^n)$, such that $m(A)>0$, is Runge,
we prove that, if $n\geq 2$, then every spirallike mapping with respect to $A$
may be approximated locally uniformly on $\Bn$ by automorphisms of $\Cn$ whose
restrictions to $\Bn$ are spirallike with respect to $A$.
In particular, any starlike mapping on $\Bn$ may be approximated
locally uniformly on $\Bn$ by automorphisms of $\Cn$ whose restrictions to $\Bn$
are starlike, for all $n\geq 2$.
Also, we prove that, if $n\ge 2$, then any convex (univalent) mapping on $\Bn$
may be approximated
locally uniformly on $\Bn$ by automorphisms of $\Cn$ whose restrictions to $\Bn$
are convex. This result is in contrast to the case of
the unit polydisc in $\Cn$, where a similar approximation property does not hold.
Similar properties also hold in the case of
$g$-starlike mappings
on $\B^n$, where $g$ is a univalent function on the unit disc $\U$ such that $g(0)=1$
and $\Re g(\zeta)>0$, $\zeta\in \U$, $n\geq 2$.
On the other hand, we also prove that in dimension $n\geq 2$,
the family of normalized automorphisms of $\Cn$
whose restrictions to
$\Bn$ have $A$-parametric representation is dense in the family $S_A^0(\Bn)$, where
$A\in L(\Cn)$ is a nonresonant operator such that $m(A)>0$.
A similar property also
holds in the case of the reachable family
$\widetilde{\mathcal R}_T({\rm id}_{\Bn},{\mathcal N}_A)$
generated by the Carath\'eodory
family ${\mathcal N}_A$, where $T>0$.
These results are generalizations
of recent results in \cite{I1} and \cite{S}.
However, the method that
we use in this paper is different from those used in the above works.

In the last part of this paper, we obtain an approximation property by automorphisms of $\C^n$
($n\geq 2$) for a subset of $S_{I_n}^0(\B^n)$ consisting of mappings $f$ which satisfy the condition
$\|Df(z)-I_n\|<1$, $z\in\B^n$. Related results will
be also obtained.

The main results of this paper can be summarized as follows.
The notations will be explained in the next sections.

\begin{theorem}
\label{introt2}
Let $A\in L(\Cn)$ be such that $m(A)>0$. Then
$$\widehat{S}_A(\Bn)=\overline{\widehat{S}_A(\Bn)\cap \mathcal{A}(\Bn)},\quad \forall\, n\geq 2.$$
In particular,
$S^*(\Bn)=\overline{S^*(\Bn)\cap \mathcal{A}(\Bn)}$, $\forall\,n\geq 2.$
\end{theorem}

\begin{theorem}
\label{introt3}
If $n\geq 2$, then $K(\Bn)=\overline{K(\Bn)\cap \mathcal{A}(\Bn)}$.
\end{theorem}

\begin{theorem}
\label{t0.1}
Let $A\in L(\Cn)$ be a nonresonant operator with $m(A)>0$. If $n\ge 2$, then
$$\ov{S^0_A(\Bn)}=\ov{S^0_A(\Bn)\cap \mathcal{A}(\Bn)}$$ and
$$\widetilde{\mathcal{R}}_T({\rm id}_{\Bn},\na)=\ov{\widetilde{\mathcal{R}}_T({\rm id}_{\Bn},\na)
\cap \mathcal{A}(\Bn)},\quad \forall\, T\in (0,\infty).$$
\end{theorem}

\begin{theorem}
\label{t0.2}
If $n\geq 2$, then
$${\mathcal Q}(\B^n)=\ov{{\mathcal Q}(\B^n)\cap\mathcal{A}(\Bn)}\, \mbox{ and }\,
\widetilde{\mathcal Q}(\B^n)=\ov{\widetilde{\mathcal Q}(\B^n)\cap \mathcal{A}(\Bn)}.$$
\end{theorem}

\section{Preliminaries}
\label{prelim}
\setcounter{equation}{0}
Let $\mathbb{C}^n$ be the space of $n$ complex
variables $z=(z_1,\dots,z_n)$ with the Euclidean inner product
$\langle z,w\rangle=\sum_{j=1}^n z_j\overline{w}_j$ and the
Euclidean norm $\|z\|=\langle z,z\rangle^{1/2}$. Also, let $\Bn$ be the
Euclidean unit ball in $\C^n$ and let $\B^1=\U$ be the unit disc.
Let $L(\C^n)$ be the space of linear operators from $\C^n$ into $\C^n$
with the standard operator norm, and let $I_n$ be the identity in $L(\C^n)$.

We denote by $H(\B^n)$ the family
of holomorphic mappings from $\B^n$ into $\mathbb{C}^n$ with the standard topology
of locally uniform convergence.
If $f\in H(\B^n)$, we say that $f$ is
normalized if $f(0)=0$ and $Df(0)=I_n$. Let ${\mathcal L}S(\Bn)$ be the subset
of $H(\Bn)$ consisting of all normalized locally biholomorphic mappings on $\Bn$, and
let $S(\B^n)$ be the subset of $H(\Bn)$ consisting
of all normalized univalent (biholomorphic) mappings on $\B^n$. Also, let
$S^*(\Bn)$ be the subset of $S(\Bn)$ consisting of starlike mappings with respect
to the origin, and let $K(\Bn)$ be the family of normalized univalent mappings
which are convex on $\Bn$.

We use the following notations related to an operator $A\in L(\Cn)$ (cf. \cite{RS}):
\begin{eqnarray*}
m(A)&=&\min\{\Re\langle A(z),z\rangle:\, \|z\|=1\},\\
k(A)&=&\max\{\Re\langle A(z),z\rangle:\, \|z\|=1\},\\
|V(A)|&=&\max\{|\langle A(z),z\rangle|:\, \|z\|=1\},\\
k_+(A)&=&\max\{\Re\lambda:\, \lambda\in \sigma(A)\},
\end{eqnarray*}
where $\sigma(A)$ is the spectrum of $A$. Note that $|V(A)|$ is the
numerical radius of the operator $A$
and $k_+(A)$ is the upper exponential index (Lyapunov index) of $A$. The
following relations hold (see e.g. \cite{GHKK-JAM}):
$$m(A)\leq k_+(A)\leq |V(A)|\leq \|A\|.$$
Also, it is known that $\|A\|\leq 2|V(A)|$
and $k_+(A)=\lim_{t\to\infty}\frac{\log\|e^{tA}\|}{t}$ (see e.g. \cite{RS}).

If $A\in L(\Cn)$ with $m(A)>0$, then
(\cite[Lemma 2.1]{DGHK}; see also \cite{GHKK-JAM}):
\begin{equation}
\label{growth-exp}
e^{m(A)t}\leq \|e^{tA}u\|\leq e^{k(A)t},\quad t\geq 0,\quad \|u\|=1.
\end{equation}

We consider the following notations (see \cite{ABW-Proc}):
$${\rm Aut}(\Cn)=\big\{\Phi:\Cn\to\Cn:\,\Phi\mbox{ is  an automorphism of }\Cn\big\},$$
$$\mathcal{A}(\Bn)=\big\{\Phi\big|_{\Bn}:\,\Phi\in {\rm Aut}(\Cn)\big\},$$
$$S_R(\Bn)=\big\{f\in S(\Bn):\,f(\Bn)\mbox{ is Runge}\big\}.$$

Let $A\in L(\Cn)$ be such that $m(A)\ge0$.
The following subsets of $H(\B^n)$ are generalizations to
$\C^n$ of the Carath\'eodory family on $\U$ (see e.g. \cite{Suf}):
$$\mathcal{N}_A=\big\{h\in H(\B^n):h(0)=0,\, Dh(0)=A,\,\Re\langle h(z),z\rangle\geq 0,\, z\in \Bn\big\},$$
and let ${\mathcal M}={\mathcal N}_{I_n}$.
These families play basic roles in the study
of Loewner chains and the associated Loewner differential equations in higher dimensions
(see \cite{GHK02}, \cite{GHKK-JAM}, \cite{GHKK13}, \cite{H11}, \cite{Ha2},
\cite{I1}, \cite{Pf}, \cite{Por1}, \cite{Roth-pontryagin},
\cite{S2}, \cite{Suf}, \cite{Vo}).

The following result will be useful in the next section. Note that the estimates
(\ref{growth1}) are due to Pfaltzgraff \cite{Pf} and Gurganus \cite{Gur}, while the growth
result (\ref{growth2}) was obtained in \cite{GHKK-JAM} (see also \cite{GHK02}).

\begin{proposition}
\label{NA-growth}
Let $A\in L(\Cn)$ be such that $m(A)\ge0$ and let $h\in\mathcal{N}_A$. Then the
following relations hold:
\begin{equation}
\label{growth1}
m(A)\|z\|^2\frac{1-\|z\|}{1+\|z\|}\le\Re\left\langle h(z),z\right\rangle
\le k(A)\|z\|^2\frac{1+\|z\|}{1-\|z\|},\quad z\in\Bn,
\end{equation}
and
\begin{equation}
\label{growth2}
\|h(z)\|\le\frac{4\|z\|}{(1-\|z\|)^2}|V(A)|,\quad z\in\Bn.
\end{equation}
\end{proposition}

A direct consequence of (\ref{growth2}) is the following compactness result
of the family ${\mathcal N}_A$ (see \cite{GHKK-JAM}; cf. \cite{GHK02}).

\begin{corollary}
\label{c.compact}
Let $A\in L(\Cn)$ be such that $m(A)\ge0$ and let $h\in\mathcal{N}_A$.
Then $\mathcal{N}_A$ is compact subset of $H(\Bn)$.
\end{corollary}

Next, we recall the definition of spirallikeness with respect to a given operator $A\in L(\Cn)$
with $m(A)>0$ (see \cite{Suf}).

\begin{definition}
\label{defin-spiral}
Let $A\in L(\Cn)$ be such that $m(A)>0$.
A mapping $f\in S(\Bn)$ is said to be spirallike with respect to $A$
(denoted by $f\in \widehat{S}_A(\Bn)$)
if $f(\Bn)$ is a spirallike domain
with respect to $A$,
i.e. $e^{-tA}f(\Bn)\subseteq f(\Bn)$, for all $t\ge0$.
\end{definition}

The following result due to Suffridge \cite{Suf} (cf. \cite{Gur}) provides a necessary and
sufficient condition of spirallikeness for locally biholomorphic mappings
on $\Bn$.

\begin{proposition}
\label{spiral}
Let $A\in L(\Cn)$ be such that $m(A)>0$, and let $f\in {\mathcal L}S(\Bn)$.
Then $f\in \widehat{S}_A(\Bn)$ iff
there exists $h\in\mathcal{N}_A$ such that $Df(z)h(z)=Af(z)$, $z\in\Bn$.
\end{proposition}

Hamada \cite{Ha2} proved the following property of spirallike mappings in $\Cn$
(see \cite{Ka}, in the case $A=I_n$).

\begin{proposition}
\label{p.spiral-runge}
Let $A\in L(\Cn)$ be such that $m(A)>0$. Then $\widehat{S}_A(\Bn)\subseteq S_R(\Bn)$.
\end{proposition}

The following subset of $S^*(\Bn)$ will occur in a forthcoming section (see e.g. \cite{HH08}).

\begin{definition}
\label{d.gstar}
Let $g:\U\to\C$ be a univalent function with $g(0)=1$ and $\Re g(\zeta)>0$ for $\zeta\in\U$.
Also, let $f\in {\mathcal L}S(\Bn)$. We say that $f$ is $g$-starlike (denoted by $f\in S_g^*(\Bn)$)
if (see \cite{GHK02})
$$\frac{1}{\|z\|^2}\langle [Df(z)]^{-1}f(z),z\rangle\in g(\U),\quad z\in\Bn\setminus\{0\}.$$
\end{definition}

\begin{remark}
\label{r.sgstar}
(i) Clearly, if $\Re g(\zeta)>0$, $\zeta\in\U$, then
$S_g^*(\B^n)\subseteq S^*(\B^n)$, and if $g(\zeta)=\frac{1-\zeta}{1+\zeta}$,
$\zeta\in\U$, then $S_g^*(\B^n)=S^*(\B^n)$.

(ii) Let $\alpha\in (0,1)$. If $g(\zeta)=\frac{1-\zeta}{1+(1-2\alpha)\zeta}$, $\zeta\in\U$,
then the family $S_g^*(\B^n)$ is the usual family $S_\alpha^*(\B^n)$ of starlike mappings of
order $\alpha$, that is
$$S_\alpha^*(\B^n)=\Big\{f\in {\mathcal L}S(\B^n):\Big|\frac{1}{\|z\|^2}
\langle [Df(z)]^{-1}f(z),z\rangle-\frac{1}{2\alpha}\Big|<\frac{1}{2\alpha},
z\in\B^n\setminus\{0\}\Big\}.$$
It is known that $K(\Bn)\subseteq S_{1/2}^*(\Bn)$ (see
\cite{GHK02}). This inclusion relation provides a motivation for the study of the family
$S_g^*(\Bn)$.

(iii) Let $\alpha\in (0,1]$.
If $g(\zeta)=\left(\frac{1-\zeta}{1+\zeta}\right)^\alpha$, $\zeta\in\U$,
then the family $S_g^*(\B^n)$ is the family $SS^*_{\alpha}(\B^n)$
of strongly starlike mappings of order
$\alpha$ on $\B^n$ (see e.g. \cite{GK}, \cite{HH08}).

(iv) If $g:\U\to\C$ is a univalent function such that
$g(0)=1$ and $\Re g(\zeta)>0$ for $\zeta\in\U$, then
the family $S_g^*(\Bn)$ is compact, by
\cite[Theorem 2.17]{DGHK}.
\end{remark}

\begin{definition}
\label{Cm}
Let $J\subseteq [0,\infty)$ be an interval.
A mapping $h:\Bn\times J\to \Cn$ is called a Carath\'{e}odory mapping on $J$ with values in $\na$ if
the following conditions hold:
\begin{item}
$(i)$ $h(\cdot,t)\in\na$, for all $t\in J$.
\end{item}
\begin{item}
$(ii)$ $h(z,\cdot)$ is measurable on $J$, for all $z\in \Bn$.
\end{item}

Let $\mathcal{C}(J,\na)$ be the family of Carath\'{e}odory mappings on $J$ with values in $\na$.

A mapping $f\in {\mathcal C}([0,\infty),{\mathcal N}_A)$ is also called a
Herglotz vector field (cf. \cite{BCM}, \cite{DGHK}).
\end{definition}

Next, we recall the notion of an $A$-normalized subordination chain on $\Bn\times [0,\infty)$,
where $A\in L(\Cn)$ with $m(A)>0$
(see \cite{GHKK-JAM}; cf. \cite{Pf}).

\begin{definition}
\label{Ldef}
A mapping $f:\Bn\times [0,\infty)\to\Cn$ is called a subordination
chain if $f(\cdot,t)\in H(\Bn)$, $f(0,t)=0$, for $t\geq 0$, and for
all $t\geq s\geq 0$
there is a holomorphic Schwarz mapping $v_{s,t}:\Bn\to\Bn$, called
the transition mapping associated with $f$, such that
$f(z,s)=f(v_{s,t}(z),t)$ for $z\in\Bn$.

A subordination chain $f$ is said to be univalent
if $f(\cdot,t)$ is a univalent mapping on $\Bn$, for all $t\geq 0$.

A subordination chain $f$
is said to be $A$-normalized
if $Df(0,t)=e^{tA}$ for $t\ge0$, where $A\in L(\Cn)$ with $m(A)>0$.

If $f:\Bn\times [0,\infty)\to\Cn$ is a univalent subordination chain, we denote the set
$\bigcup_{t\ge0}f(\Bn,t)$ by $R(f)$, and we call it the Loewner range of $f$.
\end{definition}

\begin{definition}
\label{d.s1}
Let $A\in L(\Cn)$ be such that $m(A)>0$. Also,
let $S^1_A(\Bn)$ be the family of all mappings $f\in S(\Bn)$
for which there is an $A$-normalized univalent subordination chain $L$
such that $f=L(\cdot,0)$ and $R(L)=\Cn$. If $A=I_n$, the family $S_{I_n}^1(\Bn)$ is denoted by $S^1(\Bn)$.
\end{definition}

\begin{definition} (cf. \cite{BCM}, \cite{DGHK})
\label{Lsol}
Let $A\in L(\Cn)$ be such that $m(A)>0$ and let $h\in\mathcal{C}([0,\infty),\na)$. Let $f:\Bn\times[0,\infty)\to\Cn$
be such that $f(\cdot,t)\in H(\Bn)$, $f(0,t)=0$, for $t\ge0$, and $f(z,\cdot)$ is locally absolutely continuous on $[0,\infty)$
locally uniformly with respect to $z\in\Bn$. If $f$ satisfies the Loewner differential equation
\begin{equation}
\label{Leq}
\frac{\partial f}{\partial t}(z,t)=Df(z,t)h(z,t),\quad z\in\Bn, \mbox{ a.e. }t\ge0,
\end{equation}
then $f$ is called a standard solution of $(\ref{Leq})$ associated to $h$.
\end{definition}

\begin{remark} (see \cite{Ar}, \cite{ABHK}, \cite{MV}; cf. \cite{DGHK}, \cite{GK})
\label{Lpde}

$(i)$ Let $A\in L(\Cn)$ be such that $m(A)>0$. If $f$ is an $A$-normalized univalent subordination chain, then there exists $h\in\mathcal{C}([0,\infty),\na)$ such that $f$ is a standard solution of $(\ref{Leq})$ associated to $h$.

$(ii)$ Conversely, let $h\in\mathcal{C}([0,\infty),\na)$. Then there exists an $A$-normalized univalent subordination chain $f:\Bn\times[0,\infty)\to\Cn$ that  is a standard solution of $(\ref{Leq})$ associated to $h$.  Moreover, if $g:\Bn\times[0,\infty)\to\Cn$ is another standard solution of $(\ref{Leq})$ associated to $h$, then $g$ is a subordination chain and there exists a holomorphic mapping $\Phi: R(f)\to\Cn$ such that $g=\Phi\circ f$.
\end{remark}

\begin{proof}
$(i)$ By \cite[Proposition 1.3.6]{MV}, there exists $h:\Bn\times[0,\infty)\to\Cn$ such that $h(z,\cdot)$ is measurable on $[0,\infty)$, $h(0,t)=0$ and $\Re\left\langle h(z,t),z\right\rangle\ge0$, for $t\ge0$, $z\in\Bn$, and $f$ satisfies the Loewner differential equation
(\ref{Leq}) associated with $h$.
The fact that $Dh(0,t)=A$, for a.e. $t\ge0$, is obvious in view of $(\ref{Leq})$.

$(ii)$ The existence of an $A$-normalized univalent subordination chain $f:\Bn\times[0,\infty)\to\Cn$ that  is a standard solution of $(\ref{Leq})$ associated to $h$, follows from \cite[Theorem 1.6.11]{MV} (see also \cite{Ar},\cite{Vo}). Now, let $(v_{s,t})_{t\ge s\ge 0}$ be the family of transition mappings associated to $f$. We have that $(v_{s,t})_{t\ge s\ge 0}$ satisfies the initial value problem
\begin{equation}
\label{lode2}
\displaystyle\frac{\partial }{\partial t}v_{s,t}(z)=-h(v_{s,t}(z),t),\mbox{ a.e. }t\ge s,\quad v(z,s,s)=z,
\end{equation}
for all $s\ge0$ and $z\in\Bn$ (see \cite[Theorem 5.2]{ABHK}; cf. \cite[Corollary 8.1.10]{GK}). If $g:\Bn\times[0,\infty)\to\Cn$
is another standard solution of $(\ref{Leq})$ associated to $h$, then $g(v_{s,t}(z),t)=g(z,s)$,
for $z\in\Bn$, $t\ge s\ge 0$ (use $(\ref{Leq})$ and $(\ref{lode2})$; see the proofs of  \cite[Theorem 5.2]{ABHK},
\cite[Corollary 8.1.10]{GK}). Hence $g$ is a subordination chain that admits $(v_{s,t})_{t\ge s\ge 0}$ as transition mappings.
In view of \cite[Proposition 1.5.4]{MV} (see also \cite[Theorem 4.7]{ABHK}), there is a holomorphic mapping
$\Phi: R(f)\to\Cn$ such that $g=\Phi\circ f$, as desired.
 \end{proof}

\begin{definition} (see \cite{GHKK-JAM})
\label{SA-d}
Let $A\in L(\Cn)$ be such that $m(A)>0$. Also, let $f\in H(\Bn)$ be a normalized
mapping. We say that $f$ has $A$-parametric representation
if there exists a
mapping $h\in {\mathcal C}([0,\infty), {\mathcal N}_A)$ such that
\[
f=\lim_{t\to\infty}e^{tA} v(\cdot,t)
\]
locally uniformly on $\Bn$, where
$v(z,\cdot)$ is the unique locally absolutely
continuous solution on $[0,\infty)$ of the
initial value problem
\[
\frac{d v}{d t}=-h(v,t),\mbox{ a.e. }t\ge 0,\quad v(z,0)=z,
\]
for all $z\in \Bn$.

Let $S^0_A(\Bn)$ be the family of mappings which have $A$-parametric representation on $\Bn$.
In the case $A=I_n$, let $S^0(\Bn)=S^0_{I_n}(\Bn)$ be the family of mappings with the usual parametric
representation on $\Bn$.
\end{definition}

The following definition is related to resonances/nonresonances of an operator
$A\in L(\Cn)$ (see e.g. \cite[p. 180-181]{Arn}).

\begin{definition}
\label{reson}
Let $A\in L(\Cn)$ and $(\lambda_1,\ldots,\lambda_n)\in\Cn$ be the $n$-tuple of eigenvalues of
$A$ (not necessarily distinct or in a specific order).

If there are $s\in\{1,\ldots,n\}$ and $m_1,\ldots,m_n\in\mathbb{N}\cup\{0\}$ such that
$\sum_{j=1}^nm_j\ge 2$ and
$\lambda_s=\sum_{j=1}^nm_j\lambda_j$, then we say that $A$ is resonant.
Otherwise, we say that $A$ is nonresonant.

Similarly, if there are $s\in\{1,\ldots,n\}$ and $m_1,\ldots,m_n\in\mathbb{N}\cup\{0\}$ such
that $\sum_{j=1}^nm_j\ge 2$ and
$\Re\lambda_s=\sum_{j=1}^nm_j\Re\lambda_j$,
then we say that $A$ has real
resonances;
otherwise, we say that $A$ has no real resonances.
\end{definition}

\begin{remark}
\label{spir2}
(i) Let $A\in L(\Cn)$ be such that $m(A)>0$, and let $h\in\na$. Voda \cite[Remark 3.2]{Vo}
(cf. \cite[Corollary 4.8]{DGHK} and \cite[Sect. IV.2]{P})
proved that the equation
$$Df(z)h(z)=Af(z),\quad z\in\Bn,$$
has a unique normalized solution $f\in H(\Bn)$ if and only if $A$ is nonresonant.

(ii) In view of Remark \ref{spir2} (i), we deduce the following characterization of the
Carath\'eodory family $\na$, in the case that $A\in L(\Cn)$ is a nonresonant operator
with $m(A)>0$:
$$\na=\Big\{h\in H(\Bn):h(z)=\big(Df(z)\big)^{-1}Af(z),z\in\Bn,
\mbox{for some}f\in\widehat{S}_A(\Bn)\Big\}.$$
\end{remark}

Next, we point out some examples of nonresonant operators $A\in L(\C^2)$
such that $k_+(A)=2m(A)$ or the condition
$k_+(A)<2m(A)$ does not hold.

\begin{example}
\label{ex2r}
Let $A\in L(\C^2)$ be given by
$$\begin{matrix}
A=\left(\begin{array}{cc}
\lambda_1 & 0   \\
0 & \lambda_2 \\
\end{array}
\right)
\end{matrix},$$
where $\lambda_2\ge 2\lambda_1>0$ and $\frac{\lambda_2}{\lambda_1}\notin\mathbb{N}$.
Then $A$ is a nonresonant operator with $m(A)>0$
for which the condition $k_+(A)<2m(A)$ does not hold.
\end{example}

\begin{proof}
It suffices to check that $A$ is nonresonant, since the other conditions
are clearly satisfied by $A$. Suppose that $A$ is resonant. Then there exist $s\in\{1,2\}$
and $m_1,m_2\in\mathbb{N}\cup\{0\}$ such that $m_1+m_2\ge 2$ and
$\lambda_s=m_1\lambda_1+m_2\lambda_2$. Since $0<\lambda_1<\lambda_2$ and $m_1+m_2\ge 2$,
we must have $\lambda_2=m_1\lambda_1$, which is a contradiction
with $\frac{\lambda_2}{\lambda_1}\notin\mathbb{N}$. This completes the proof.
 \end{proof}

\begin{example}
\label{ex3r}
Let $A\in L(\C^2)$ be given by
$$\begin{matrix}
A=\left(\begin{array}{cc}
1-2\alpha & 1   \\
0 & \frac{1}{2}-2\alpha \\
\end{array}
\right)
\end{matrix},$$
where $\alpha=\min_{t\in[0,1]}\big(\frac{t}{2}-\sqrt{t(1-t)}\big)$.
Then $A$ is a nonresonant operator with $m(A)>0$ and
$k_+(A)=2m(A)$.
\end{example}

\begin{proof}
By elementary computations, we obtain that $\alpha=\frac{1-\sqrt{5}}{4}<0$. Also, it is not difficult to prove that $m(A)=\frac{1}{2}-\alpha>0$. Since $\sigma(A)=\{1-2\alpha, \frac{1}{2}-2\alpha\}$, we have $k_+(A)=1-2\alpha$ and thus $k_+(A)=2m(A)$. If $A$ is resonant, then we deduce
by the same arguments as in the proof of Example $\ref{ex2r}$
that there exists $m\in\mathbb{N}$ such that $m\ge2$ and $1-2\alpha=m(\frac{1}{2}-2\alpha)$.
 However, this is a contradiction.
 \end{proof}

\begin{remark}
\label{rrr}
$(i)$ Let $A\in L(\Cn)$ be a positive definite Hermitian matrix such that $k_+(A)\le 2m(A)$.
Then $m(A)>0$ and $A$ is resonant if and only if $k_+(A)=2m(A)$.

$(ii)$ Let $A\in L(\Cn)$ be such that $m(A)>0$ and $k_+(A)\le 2m(A)$.
Then $A$ has real resonances if and only if $k_+(A)=2m(A)$.
\end{remark}

\begin{proof}
$(i)$ Since $A$ is a positive definite Hermitian matrix, we have that $m(A)>0$
is the smallest eigenvalue of $A$ (see e.g. \cite{RS}). Since $k_+(A)$
is the largest eigenvalue of $A$, it follows easily that $A$ is resonant if and only if $k_+(A)=2m(A)$.

$(ii)$ Let $k_-(A)=\min\{\Re\lambda:\, \lambda\in \sigma(A)\}$. Then we have
$$m(A)\le k_-(A)\le k_+(A)\le 2 m(A).$$
Now, it is easy to see that $A$ has real resonances if and only if $k_+(A)=2m(A)$.
This completes the proof.
 \end{proof}

\section{Reachable families associated with the Carath\'eodory family ${\mathcal N}_A$}
\label{sec:reachable}
\setcounter{equation}{0}
In this section, we consider the notion of a reachable family associated with the Carath\'eodory
family ${\mathcal N}_A$ (see \cite{GHKK13}). We obtain a characterization of this
family in terms of $A$-normalized univalent subordination chains, and we prove that it is a
compact subset of $H(\B^n)$, for all $A\in L(\C^n)$ with $m(A)>0$. These results are generalizations
of recent results obtained in \cite{GHKK13}, in the case $k_+(A)<2m(A)$.

\begin{definition}
\label{Rf}
Let $T\in(0,\infty)$ and let $A\in L(\Cn)$ be such that $m(A)>0$.
For every $h\in\mathcal{C}([0,T],\na)$, let $v=v(z,t;h)$
be the unique locally absolutely continuous solution on $[0,T]$ of the initial value problem
\begin{equation}
\label{lode}
\displaystyle\frac{\partial v}{\partial t}=-h(v,t),\mbox{ a.e. }t\in [0,T],\quad  v(z,0;h)=z,
\end{equation}
for all $z\in \Bn$. Also, let
$$\widetilde{\mathcal{R}}_T({\rm id}_{\Bn},\na)=\big\{e^{TA}v(\cdot,T;h):\, h\in\mathcal{C}([0,T],\na)\big\},$$
be the {\it time-$T$-reachable family} of $(\ref{lode})$.
\end{definition}

\begin{remark}
\label{rrf}
Let $T>0$ and $A\in L(\Cn)$ be such that $m(A)>0$.
Then $\widetilde{\mathcal{R}}_T({\rm id}_{\Bn},\na)\subseteq S^0_A(\Bn)$
(cf. \cite{GHKK13}, in the case $k_+(A)<2m(A)$).
\end{remark}

\begin{proof}
Let $f\in \widetilde{\mathcal{R}}_T({\rm id}_{\Bn},\na)$.
Then there exists $h\in \mathcal{C}([0,T],\na)$ such that $f=e^{TA}v(\cdot,T;h)$.
Let $\widetilde{h}:\Bn\times[0,\infty)\to\Cn$ be given by
$$\widetilde{h}(z,t)=\left\{
\begin{array}{ll}
h(z,t),&t\in[0,T],\, z\in\Bn\\
Az,&t>T,\, z\in\Bn.
\end{array}
\right.$$
It is easily seen that $\widetilde{h}\in \mathcal{C}([0,\infty),\na)$
and
$$v(z,t;\widetilde{h})=\left\{
\begin{array}{ll}
v(z,t;h),&t\in[0,T],\, z\in\Bn\\
e^{(T-t)A}v(z,T;h),&t>T,\, z\in\Bn.
\end{array}
\right.$$
Hence $f=e^{TA}v(\cdot,T;h)=\lim_{t\to\infty}e^{tA}v(\cdot,t;\widetilde{h})$, and thus
$f\in S^0_A(\Bn)$, as desired.
 \end{proof}

We obtain the following characterization of the
family $\widetilde{\mathcal{R}}_T({\rm id}_{\Bn},\na)$ in terms of $A$-normalized
univalent subordination chains (see \cite{GHKK13}, in the case $k_+(A)<2m(A)$).

\begin{proposition}
\label{reachci}
Let $T\in(0,\infty)$ and $A\in L(\Cn)$ be such that $m(A)>0$. Then
$\varphi\in\widetilde{\mathcal{R}}_T({\rm id}_{\Bn},\na)$ if and only if there is an $A$-normalized univalent subordination chain $f$ such that $f(\cdot,0)=\varphi$, and  $f(\cdot,t)=e^{tA}{\rm id}_{\Bn}$, for $t\ge T$.
\end{proposition}

\begin{proof}
In the following, we shall use arguments similar to those in the proof of
\cite[Theorem 4.5]{GHKK13}. Let $\varphi\in\widetilde{\mathcal{R}}_T({\rm id}_{\Bn},\na)$. Then there is $h\in\mathcal{C}([0,T],\mathcal{N}_A)$ such that $\varphi=e^{TA}v(\cdot,0,T;h)$, where $v(z,s,\cdot;h)$ is the unique locally absolutely continuous solution on $[s,T]$ of the initial value problem associated to $h$ (see \cite[Theorem 2.1]{GHKK-JAM}):
$$\frac{\partial v}{\partial t}=-h(v,t),\mbox{ a.e. }t\in [s,T],\quad  v(z,s,s;h)=z,$$
for all $z\in\Bn$ and $s\in[0,T)$. We define $f:\Bn\times[0,\infty)\to\Cn$ by
$$
f(z,t)=\left\{
\begin{array}{ll}
e^{TA}v(z,t,T;h),&0\le t\le T\\
e^{tA}z,&t>T.
\end{array}
\right.$$
Since $v(\cdot,s,T;h)$ is a Schwarz univalent mapping, for all $s\in[0,T]$, and $v$ satisfies the semigroup property, $v(z,s,T)=v(v(z,s,t),t,T)$, for $0\le s\le t$, we deduce that $f$ is an $A$-normalized univalent subordination chain such that $f(\cdot,0)=\varphi$ and  $f(\cdot,t)=e^{tA}{\rm id}_{\Bn}$, for $t\ge T$.

For the other implication, let $f$ be an $A$-normalized univalent subordination chain such that $f(\cdot,0)=\varphi$ and  $f(\cdot,t)=e^{tA}{\rm id}_{\Bn}$, for $t\ge T$. By Remark $\ref{Lpde}$ (i), there exists $h\in\mathcal{C}([0,\infty),\na)$ such that $f$ is a standard solution of $(\ref{Leq})$ associated to $h$. Let $(v_{s,t})_{t\ge s\ge 0}$ be the family of transition mappings associated to $f$. $(v_{s,t})_{t\ge s\ge 0}$ satisfies the initial value problem $(\ref{lode2})$ associated to $h$ (see the proof of Remark $\ref{Lpde}$ (ii)). Since $\varphi(z)=f(z,0)=f(v_{0,T}(z),T)=e^{TA}v_{0,T}(z)$, $z\in\Bn$, we have $\varphi\in\widetilde{\mathcal{R}}_T({\rm id}_{\Bn},\na)$. This completes the proof.
 
\end{proof}

The following compactness result of the family $\widetilde{\mathcal{R}}_T({\rm id}_{\Bn},\na)$
is a generalization of \cite[Corollary 4.7]{GHKK13}.

\begin{proposition}
\label{reachcii}
Let $T\in(0,\infty)$ and let $A\in L(\Cn)$ be such that $m(A)>0$. Then
$\widetilde{\mathcal{R}}_T({\rm id}_{\Bn},\na)$ is a compact set in $H(\Bn)$.
\end{proposition}

\begin{proof}
By \cite[Theorem 2.1, (2.2)]{GHKK-JAM} and $(\ref{growth-exp})$, we deduce that, for every $t\in[0,T]$ and
$f\in\widetilde{\mathcal{R}}_t({\rm id}_{\Bn},\na)$, we have
$$\|f(z)\|\le\|e^{tA}\|e^{-m(A)t}\frac{\|z\|}{(1-\|z\|)^2}\le e^{T(k(A)-m(A))}\frac{\|z\|}{(1-\|z\|)^2},\quad z\in\Bn.$$
In particular, $\widetilde{\mathcal{R}}_T({\rm id}_{\Bn},\na)$ is a normal family.

Next, we use arguments  similar to those in the proof of  \cite[Corollary 4.7]{GHKK13}, to deduce that $\widetilde{\mathcal{R}}_T({\rm id}_{\Bn},\na)$ is closed. Let $(\varphi_k)_{k\in\N}$ be a sequence in $\widetilde{\mathcal{R}}_T({\rm id}_{\Bn},\na)$ that converges, locally uniformly on $\Bn$, to $\varphi\in H(\Bn)$.  In view of Proposition $\ref{reachci}$, for every $k\in\N$, there exists an $A$-normalized univalent subordination chain $f_k$ such that $f_k(\cdot,0)=\varphi_k$ and  $f_k(\cdot,t)=e^{tA}{\rm id}_{\Bn}$, for $t\ge T$. For every $t\in[0,T]$ and $k\in\N$, let $F_{t,k}(z,s)=e^{-tA}f_k(z,t+s)$, for $z\in\Bn$, $s\ge0$. Then $F_{t,k}$ is an $A$-normalized univalent subordination chain such that $F_{t,k}(\cdot,s)=e^{sA}{\rm id}_{\Bn}$, for $s\ge T-t$, and thus, by Proposition \ref{reachci}, $F_{t,k}(\cdot, 0)=e^{-tA}f_k(\cdot,t)\in\widetilde{\mathcal{R}}_{T-t}({\rm id}_{\Bn},\na)$, for all $t\in[0,T]$ and $k\in\N$. From the above inequality we deduce that, for every $r\in(0,1)$, there exists $C_r>0$ such that
$$\|e^{-tA}f_k(z,t)\|\le C_r,\qquad \|z\|\le r,\,t\ge 0,\,k\in\N.$$
Using arguments similar to those in the proof of \cite[Theorem 8.1.14]{GK},
we deduce that there exists a subsequence $(f_{k_p})_{p\in\N}$ such that $f_{k_p}(\cdot,t)\to f(\cdot,t)$, locally uniformly on $\Bn$, as $p\to\infty$, for all $t\ge 0$, where $f$ is an $A$-normalized univalent subordination chain such that $f(\cdot,0)=\varphi$ and  $f(\cdot,t)=e^{tA}{\rm id}_{\Bn}$, for $t\ge T$. Hence, $\varphi\in\widetilde{\mathcal{R}}_T({\rm id}_{\Bn},\na)$, in view of Proposition $\ref{reachci}$.
This completes the proof.
 \end{proof}

\section{Density results for certain subsets of $S(\Bn)$}
\label{sec:density}
\setcounter{equation}{0}
In this section we obtain approximation properties
of starlike, convex, spirallike mappings, and
mappings which have $A$-parametric representation on $\Bn$, by automorphisms of $\Cn$,
$n\geq 2$.

First, we obtain the following approximation result
of spirallike mappings (in particular, starlike mappings) on $\Bn$, by
automorphisms of $\Cn$ whose restrictions to $\Bn$ are spirallike (respectively, starlike),
if $n\geq 2$.

\begin{theorem}
\label{tspst}
Let $A\in L(\Cn)$ be such that $m(A)>0$. Then
$$\widehat{S}_A(\Bn)=\overline{\widehat{S}_A(\Bn)\cap \mathcal{A}(\Bn)},\quad\forall\,n\geq 2.$$
In particular,
$$S^*(\Bn)=\overline{S^*(\Bn)\cap \mathcal{A}(\Bn)},\quad \forall\,n\geq 2.$$
\end{theorem}

\begin{proof}
Let $f\in \widehat{S}_A(\Bn)$. By \cite[Theorem 3.1]{Ha2},
we have that $f(\Bn)$ is a Runge domain. By \cite[Theorem 2.1]{And-Lemp} (see also
\cite[Theorem 1.1]{FR}), there exists a sequence  $(\psi_k)_{k\in\mathbb{N}}$
in ${\rm Aut}(\Cn)$ that converges locally uniformly on $\Bn$ to $f$.
In view of \cite[Theorem 2.17]{DGHK}, we may assume that the mappings
in $(\psi_k)_{k\in\mathbb{N}}$ are normalized, and we may also deduce that
\begin{equation}
\label{e1}
\big(D\psi_k(z)\big)^{-1}A\psi_k(z)\to \big(Df(z)\big)^{-1}Af(z), \mbox{ as }k\to\infty,
\end{equation}
locally uniformly with respect to $z\in\Bn$.

Let $(r_m)_{m\in\mathbb{N}}$ be a sequence in $(0,1)$ that converges to $1$.
Taking into account (\ref{growth1}) and Proposition \ref{spiral}, we deduce that
\begin{equation}
\label{e2}
\Re\big\langle \big(Df(r_m z)\big)^{-1}Af(r_m z),r_m z\big\rangle\ge m(A)r_m^2\|z\|^2
\frac{1-r_m}{1+r_m}>0,
\end{equation}
for all $z\in\Bn\setminus\{0\}$ and $m\in\mathbb{N}$.
By $(\ref{e1})$ and $(\ref{e2})$, we deduce that for every $m\in\mathbb{N}$, there exists $k_m\in\mathbb{N}$
such that $$\Re\big\langle \big(D\psi_{k_m}(r_m z)\big)^{-1}A\psi_{k_m}(r_m z),r_mz\big\rangle>0,
\quad z\in\Bn\setminus\{0\}.$$
The sequence $(k_m)_{m\in\N}$ may be chosen to be increasing.
For every $m\in\mathbb{N}$, let $\varphi_m(z)=\frac{1}{r_m}\psi_{k_m}(r_mz)$, $z\in\Bn$.
Then $(\varphi_m)_{m\in\mathbb{N}}$ is a sequence in $\mathcal{A}(\Bn)\cap \widehat{S}_A(\Bn)$.
Since $(\psi_k)_{k\in\mathbb{N}}$ converges locally uniformly on $\Bn$ to $f$ and $(r_m)_{m\in\mathbb{N}}$
converges to $1$, we deduce that $(\varphi_m)_{m\in\mathbb{N}}$
converges locally uniformly on $\Bn$ to $f$.
Therefore $f\in \overline{\widehat{S}_A(\Bn)\cap \mathcal{A}(\Bn)}$.

Using Proposition \ref{spiral},
one can easily prove that $\widehat{S}_A(\Bn)$ is closed in $H(\Bn)$,
for every $A\in L(\Cn)$ with $m(A)>0$.
From this, we obtain the reverse inclusion.
This completes the proof.
 \end{proof}

In connection with Theorem \ref{tspst}, we have the following density result for $K(\Bn)$.

\begin{theorem}
\label{convaut}
If $n\geq 2$, then $K(\Bn)=\overline{K(\Bn)\cap \mathcal{A}(\Bn)}$.
\end{theorem}

\begin{proof}
Let $f\in K(\Bn)$. Since $f(\Bn)$ is convex, $f(\Bn)$ is a Runge domain (see e.g. \cite[Theorem 3.1]{Ha2}). By \cite[Theorem 2.1]{And-Lemp} (see also
\cite[Theorem 1.1]{FR}), there exists a sequence  $(\psi_k)_{k\in\mathbb{N}}$
in ${\rm Aut}(\Cn)$ that converges locally uniformly on $\Bn$ to $f$.
In view of \cite[Theorem 2.17]{DGHK}, we can assume that the mappings
in $(\psi_k)_{k\in\mathbb{N}}$ are normalized, and we can also deduce that
\begin{equation}
\label{econv}
\big(D\psi_k(z)\big)^{-1}D^2\psi_k(z)\to \big(Df(z)\big)^{-1}D^2f(z), \mbox{ as }k\to\infty,
\end{equation}
locally uniformly with respect to $z\in\Bn$.

Since $f\in K(\Bn)$, we have that (see e.g. \cite[Theorem 6.3.4]{GK})
\begin{equation}
\label{ekik}
1-\Re\big\langle \big(Df(z)\big)^{-1}D^2f(z)(v,v),z\big\rangle>0,
\end{equation}
for all $z\in\Bn$, $v\in\Cn$ with $\Re\left\langle z,v\right\rangle=0$, $\|v\|=1$.

Let $(r_m)_{m\in\mathbb{N}}$ be a sequence in $(0,1)$ that converges to $1$.
Fix an arbitrary $m\in\mathbb{N}$. Let
$$\delta_m=\max\Big\{\Re\big\langle \big(Df(z)\big)^{-1}D^2f(z)(v,v),z\big\rangle:\|z\|\le r_m, \|v\|=1, \Re\left\langle z,v\right\rangle=0\Big\}.$$
By $(\ref{ekik})$, we have $1-\delta_m>0$. In view of $(\ref{econv})$, there is $k_m\in\mathbb{N}$ such that
$$\max_{\|v\|=1}\Big\|\big(D\psi_{k_m}(z)\big)^{-1}D^2\psi_{k_m}(z)(v,v)- \big(Df(z)\big)^{-1}D^2f(z)(v,v)\Big\|<1-\delta_m,$$
for $\|z\|\le r_m$.
We denote by $\varepsilon_m$ the left hand side of the above inequality.

Let $\varphi_m(z)=\frac{1}{r_m}\psi_{k_m}(r_mz)$, $z\in\Bn$.
The sequence $(k_m)_{m\in\N}$ may be chosen to be increasing.
Taking into account the above inequalities, we deduce that
$$1-\Re\big\langle \big(D\varphi_m(z)\big)^{-1}D^2\varphi_m(z)(v,v),z\big\rangle$$
$$=1-\Re\big\langle \big(D\psi_{k_m}(r_mz)\big)^{-1}D^2\psi_{k_m}(r_mz)(v,v),r_mz\big\rangle$$
$$\ge 1-\Re\big\langle \big(Df(r_mz)\big)^{-1}D^2f(r_mz)(v,v),r_mz\big\rangle$$
$$-\Big\|\big(D\psi_{k_m}(r_mz)\big)^{-1}D^2\psi_{k_m}(r_mz)(v,v)- \big(Df(r_mz)\big)^{-1}D^2f(r_mz)(v,v)\Big\|$$
$$\ge 1-\delta_m-\varepsilon_m>0,$$
for all $z\in\Bn$, $v\in\Cn$ with $\Re\left\langle z,v\right\rangle=0$, $\|v\|=1$. Hence $\varphi_m\in K(\Bn)$.

Since $r_m\to 1$, as $m\to\infty$, the sequence $(\varphi_m)_{m\in\mathbb{N}}$ in $K(\Bn)\cap\mathcal{A}(\Bn)$
converges locally uniformly on $\Bn$ to $f$. Therefore, $f\in \ov{K(\Bn)\cap\mathcal{A}(\Bn)}$. The reversed inclusion is obvious, since $K(\Bn)$ is a compact family.
 \end{proof}

The next proposition provides a basic difference between the case of the Euclidean unit ball $\Bn$ and the case of the unit polydisc $\mathbb{U}^n$ in $\Cn$,
regarding Theorem $\ref{convaut}$. Let $K(\mathbb{U}^n)$ be the family of univalent normalized holomorphic mappings $f:\mathbb{U}^n\to\Cn$ such that $f(\mathbb{U}^n)$ is convex. Also, let $\mathcal{A}(\mathbb{U}^n)$ be the family of normalized automorphisms of $\Cn$ that are restricted to $\mathbb{U}^n$.

\begin{proposition}
\label{pcpol}
$\ov{K(\U^n)\cap \mathcal{A}(\U^n)}=\{{\rm id}_{\U^n}\}\subsetneq K(\U^n)$, for all $n\in\N.$
\end{proposition}

\begin{proof}
We have (see e.g. \cite[Theorem 6.3.2]{GK})
$$K(\mathbb{U}^n)=\{f:\mathbb{U}^n\to\Cn:f(z)=
(\varphi_1(z_1),\ldots,\varphi_n(z_n)),z=(z_1,\ldots,z_n)\in\mathbb{U}^n,$$
$$\mbox{ for some }\varphi_1,\ldots \varphi_n\in K(\mathbb{U})\}.$$
Let $\phi\in K(\mathbb{U}^n)\cap \mathcal{A}(\mathbb{U}^n)$. Then $\phi(z)=(\phi_1(z_1),\ldots,\phi_n(z_n))$, $z=(z_1,\ldots,z_n)\in\mathbb{U}^n$, for some $\phi_1,\ldots \phi_n\in K(\mathbb{U})$. Since $\phi\in\mathcal{A}(\mathbb{U}^n)$, there exists $\Phi\in {\rm Aut}(\Cn)$ such that $\Phi\big|_{\U^n}=\phi$. Let $\Phi^1,\ldots,\Phi^n:\Cn\to\C$ be the components of $\Phi$. Fix an arbitrary  $i\in\{1,\ldots,n\}$. For every $j\in\{1,\ldots,n\}$ with $i\neq j$, we have
$$\frac{\partial \Phi^i}{\partial z_j}(z)=\frac{\partial }{\partial z_j}\phi_i(z_i)=0,\mbox{ for all } z=(z_1,\ldots,z_n)\in\U^n.$$
By the identity principle for holomorphic mappings, we deduce that $\Phi^i$ depends only on $z_i$ on $\Cn$. Hence $\phi_i$ has a holomorphic extension to $\C$, given by $\Phi^i$. Let us use the same notation $\phi_i$ for this extension. Thus we have $\Phi(z)=(\phi_1(z_1),\ldots,\phi_n(z_n))$, for $z=(z_1,\ldots,z_n)\in\Cn$. Since $\Phi\in{\rm Aut}(\Cn)$, it is easy to prove that every $\phi_i$ is an automorphism of $\C$, and thus $\phi_i={\rm id}_{\C}$. Therefore $\phi={\rm id}_{\U^n}$.
This completes the proof.
 \end{proof}

The following density result related to
the family $S_g^*(\Bn)$ is in connection with Theorem \ref{convaut}.

\begin{theorem}
\label{grp}
Let $g:\U\to\C$ be a univalent function such that $g(0)=1$ and $\Re g(\zeta)>0$ for $\zeta\in\U$.
Then
$$S^*_g(\Bn)=\overline{S^*_g(\Bn)\cap\mathcal{A}(\Bn)},\quad \forall\, n\geq 2.$$
\end{theorem}

\begin{proof}
First, note that the family $S_g^*(\B^n)$ is compact, in view of Remark \ref{r.sgstar} (iv).
We shall use arguments similar to those in the proof of Theorem \ref{tspst}.
Let $f\in S^*_g(\Bn)$ and let $(\psi_k)_{k\in\mathbb{N}}$ be a sequence
in ${\rm Aut}(\Cn)$ that converges locally uniformly on $\Bn$ to $f$.
We may assume that $\psi_k$ are normalized.
First, we observe that, for every $r\in(0,1)$, there exists $\delta\in(0,1)$ such that
$\frac{1}{\|z\|^2}\langle [Df(z)]^{-1}f(z),z\rangle\in g(\delta\U)$, for $z\in r\Bn\setminus\{0\}$.
Next, for every $k\in\N$, let
$\alpha_k:\Bn\to\Cn$ be given by
$$\alpha_k(z)=[Df(z)]^{-1}f(z)-[D\psi_k(z)]^{-1}\psi_k(z),\quad z\in\Bn.$$
Taking into account the kernel convergence result given in \cite[Theorem 2.17]{DGHK},
we deduce that for every $r\in(0,1)$,
$$\frac{1}{\|z\|^2}\Big|\langle {\alpha_k(z)},{z}\rangle\Big|=
\left|\frac{1}{\|z\|^2}\Big\langle\int_0^1 D\alpha_k(tz)(z)dt,z\Big\rangle\right|\leq\sup_{\zeta\in r\Bn}\left\|D\alpha_k(\zeta)\right\|\to 0,$$
as $k\to\infty$, uniformly with respect to $z\in r\B^n\setminus\{0\}$.

Let $(r_m)_{m\in\mathbb{N}}$ be a sequence in $(0,1)$ with $r_m\to 1$.
Fix $m\in\mathbb{N}$. From the above arguments, there exists $k_m\in\N$
such that $\frac{1}{\|z\|^2}\langle \big(D\varphi_m(z)\big)^{-1}\varphi_m(z),z\rangle\in g(\U)$,
for all $z\in \Bn\setminus\{0\}$,
where $\varphi_m(z)=\frac{1}{r_m}\psi_{k_m}(r_mz)$, for $z\in\Bn$.
The sequence $(k_m)_{m\in\N}$ may be chosen to be increasing.
Hence, we obtain a sequence $(\varphi_m)_{m\in\mathbb{N}}$ in $S^*_g(\Bn)\cap\mathcal{A}(\Bn)$
that converges locally uniformly on $\Bn$ to $f$. This completes the proof.
 \end{proof}

In particular, from Theorem \ref{grp}, we obtain the following consequence:

\begin{corollary}
\label{c.grp}
$(i)$ If $\alpha\in (0,1)$, then
$S^*_\alpha(\Bn)=\overline{S^*_\alpha(\Bn)\cap\mathcal{A}(\Bn)}$, $n\geq 2$.

$(ii)$ If $\alpha\in (0,1]$, then
$SS^*_\alpha(\Bn)=\overline{SS^*_\alpha(\Bn)\cap\mathcal{A}(\Bn)}$, $n\geq 2$.
\end{corollary}

\begin{theorem}
\label{p2}
Let $T\in (0,\infty)$ and let $A\in L(\Cn)$ be a nonresonant operator with $m(A)>0$. If $n\ge 2$, then
$\widetilde{\mathcal{R}}_T({\rm id}_{\Bn},\na)=\ov{\widetilde{\mathcal{R}}_T({\rm id}_{\Bn},\na)
\cap \mathcal{A}(\Bn)}$.
\end{theorem}

\begin{proof}
Let $f\in \widetilde{\mathcal{R}}_T({\rm id}_{\Bn},\na)$. Then there is $h\in \mathcal{C}([0,T],\na)$
with $f=e^{TA}v(\cdot,T;h)$. From \cite[Lemma 4.12]{GHKK13}, the mapping
$h$ can be approximated pointwise almost everywhere on $[0,T]$ by a sequence $(h_k)_{k\in\mathbb{N}}$
of piecewise
constant Carath\'eodory mappings with values in $\na$ such that $e^{TA}v(\cdot,T;h_k)\to e^{TA}v(\cdot,T;h)$, as $k\to\infty$,
locally uniformly on $\Bn$. In view of Remark $\ref{spir2}$ and Theorem \ref{tspst}, we can choose the constants
associated to $h_k$ of the following form: $z\mapsto\big(D\alpha(z)\big)^{-1}A\alpha(z)$ with $\alpha\in\mathcal{A}(\Bn)\cap \widehat{S}_A(\Bn)$,
for all $k\in\mathbb{N}$. Then it can be easily proved that $e^{TA}v(\cdot,T;h_k)\in \mathcal{A}(\Bn)$, for all $k\in\mathbb{N}$.
Thus $f\in\ov{\widetilde{\mathcal{R}}_T({\rm id}_{\Bn},\na)\cap \mathcal{A}(\Bn)}$.

The reversed inclusion follows by Proposition \ref{reachcii}.
 \end{proof}

In view of Theorem \ref{p2}, we obtain the following density result related to the family
$\widetilde{\mathcal{R}}_T({\rm id}_{\Bn},\mathcal M)$.

\begin{corollary}
\label{c.reachable}
Assume that $T>0$ and $n\ge 2$. Then
$$\widetilde{\mathcal{R}}_T({\rm id}_{\Bn},{\mathcal M})=\ov{\widetilde{\mathcal{R}}_T({\rm id}_{\Bn},{\mathcal M})
\cap \mathcal{A}(\Bn)}.$$
\end{corollary}

The next result is a generalization of \cite[Corollary 2.3]{I1}
and \cite[Corollary 2.6.10, Question 2.6.11]{S} to the case of mappings with
$A$-parametric representation on $\Bn$.

\begin{theorem}
\label{p3}
Let $A\in L(\Cn)$ be a nonresonant operator such that $m(A)>0$. If $n\ge 2$, then
$$\ov{S^0_A(\Bn)}=\ov{S^0_A(\Bn)\cap \mathcal{A}(\Bn)}.$$
\end{theorem}

\begin{proof}
Let $f\in S^0_A(\Bn)$. Then there is $h\in \mathcal{C}([0,\infty),\na)$
with $f=\lim_{t\to\infty}e^{tA}v(\cdot,t;h)$, locally uniformly on $\Bn$.
In particular, the sequence of mappings  $(e^{mA}v(\cdot,m;h))_{m\in\mathbb{N}}$
converges locally uniformly to $f$. Since $A$ is nonresonant,
Theorem $\ref{p2}$
implies that for every $m\in\mathbb{N}$ there exists
$\varphi_m\in\widetilde{\mathcal{R}}_m({\rm id}_{\Bn},\na)\cap\mathcal{A}(\Bn)$
such that the distance (with respect to the compact open topology on $H(\Bn)$)
between $e^{mA}v(\cdot,m;h)$ and $\varphi_m$ is less than $\frac{1}{m}$.
Taking into account Remark \ref{rrf}, we have that $(\varphi_{m})_{m\in\mathbb{N}}$
is a sequence in $S^0_A(\Bn)\cap \mathcal{A}(\Bn)$ that converges locally uniformly to $f$.
Hence $S^0_A(\Bn)\subseteq\ov{S^0_A(\Bn)\cap \mathcal{A}(\Bn)}$. The reversed inclusion is obvious.
 \end{proof}

In particular, if $A\in L(\Cn)$ is such that $k_+(A)<2m(A)$, then $A$ is nonresonant,
and in view of Theorem \ref{p3},
we obtain the following consequence (see \cite[Corollary 2.3]{I1}
and \cite[Corollary 2.6.10, Question 2.6.11]{S}, in the case $A=I_n$).

\begin{corollary}
\label{c.density1}
Let $A\in L(\C^n)$ be such that $k_+(A)<2m(A)$. If $n\geq 2$, then
$$S_A^0(\Bn)=\ov{S_A^0(\Bn)\cap \mathcal{A}(\Bn)}.$$
\end{corollary}

\begin{corollary}
\label{c.density2}
Let $A\in L(\C^n)$ be such that $k_+(A)<2m(A)$. If $n\geq 2$, then
$$\ov{S_A^1(\Bn)}=\ov{S_A^1(\Bn)\cap \mathcal{A}(\Bn)}.$$
\end{corollary}

\begin{proof}
It suffices to show that $S_A^1(\B^n)\subseteq \ov{{\mathcal A}(\B^n)\cap S_A^1(\B^n)}$.
To this end, let $f\in S_A^1(\B^n)$. Then $f=\Phi\circ f_0$ for some
$f_0\in S_A^0(\B^n)$ and $\Phi\in {\rm Aut}(\C^n)$ such that $\Phi(0)=0$ and $D\Phi(0)=I_n$
(cf. \cite[Theorem 3.1]{DGHK}). Since $f_0\in S_A^0(\B^n)$, there is a sequence $(\psi_k)_{k\in\N}$ in
${\mathcal A}(\B^n)\cap S_A^0(\B^n)$ such that $\psi_k\to f_0$ locally uniformly on $\B^n$, as
$k\to\infty$, by Corollary \ref{c.density1}. Now, if $\phi_k=\Phi\circ \psi_k$, $k\in\N$, then
it is clear that $\phi_k\in {\mathcal A}(\B^n)\cap S_A^1(\B^n)$ and
$\phi_k\to f$ locally uniformly on $\B^n$, as $k\to\infty$. This completes the proof.
 
\end{proof}

\begin{remark}
\label{r1a}
Let $A\in L(\C^2)$ be as in Example \ref{ex3r}. We note that Theorems $\ref{p2}$ and $\ref{p3}$
hold for this $A$, even though $k_+(A)=2m(A)$, because $A$ is nonresonant.
\end{remark}

\begin{question}
\label{q1}
Does there exist a resonant operator $A\in L(\Cn)$, $n\geq 2$, with $m(A)>0$, for which
Theorems $\ref{p2}$ and $\ref{p3}$ remain true?
\end{question}

Next, we prove an embedding property related to the families
$S_A^1(\Bn)$, ${\mathcal A}(\Bn)$, and
$S_R(\Bn)$ (cf. \cite{ABW-Proc}, \cite{I1}, \cite{S}).

\begin{proposition}
\label{rs}
Let $n\geq 2$ and $A\in L(\Cn)$ be such that $m(A)>0$. Then
$$\mathcal{A}(\Bn)\cap S(\Bn)\subseteq S^1_A(\Bn)\subseteq S_R(\Bn)=\ov{\mathcal{A}(\Bn)\cap S(\Bn)}.$$
\end{proposition}

\begin{proof}
For the first inclusion, let $\Phi\in {\rm Aut}(\Cn)$
be such that $\Phi(0)=0$ and $D\Phi(0)=I_n$. Also, let $L:\Bn\times [0,\infty)\to \Cn$
be given by $L(z,t)=\Phi(e^{tA}z)$,
for $z\in\Bn$ and $t\ge0$. Then $L$
is an $A$-normalized univalent subordination chain. Since $\|e^{tA}z\|\ge e^{m(A)t}\|z\|$,
for $z\in\Bn$, $t\ge0$, by (\ref{growth-exp}), and since $m(A)>0$, we deduce that $R(L)=\Cn$, and thus $L(\cdot,0)=\Phi\big|_{\Bn}\in S^1_A(\Bn)$.
In view of the proof of \cite[Theorem 5.1]{Ha2} and \cite[Satz 17-19]{Doc-Gra}, we have $S^1_A(\Bn)\subseteq S_R(\Bn)$.
The inclusion $S_R(\Bn)\subseteq\ov{\mathcal{A}(\Bn)\cap S(\Bn)}$
is a consequence of \cite[Theorem 2.1]{And-Lemp} (see also \cite[Theorem 1.1]{FR}). 
On the other hand, taking into account \cite[Corollary 3.3]{ABW-Proc} 
(see also \cite[p. 372]{And-Lemp}), we deduce that 
$\ov{\mathcal{A}(\Bn)\cap S(\Bn)}\subseteq S_R(\B^n)$. 
Hence, $S_R(\B^n)=\ov{\mathcal{A}(\Bn)\cap S(\Bn)}$, as desired.
This completes the proof.
 \end{proof}

\begin{question}
Let $A\in L(\Cn)$ be such that $m(A)>0$, and let $n\geq 2$. Is it true that $S^1_A(\Bn)=S^1(\Bn)$?
\end{question}

\begin{remark}
\label{r.a-parametric}
(i) From Proposition \ref{rs}, we deduce that if $n\geq 2$ and $A\in L(\C^n)$ with $m(A)>0$, then
$\ov{S_A^1(\B^n)}=\ov{{\mathcal A}(\B^n)\cap S(\B^n)}$.
Since $\ov{S^1(\B^n)}=\ov{{\mathcal A}(\B^n)\cap S(\B^n)}$ (cf. \cite{ABW-Proc},
\cite{I1}, \cite{S2}), it follows that
$\ov{S_A^1(\B^n)}=\ov{S^1(\B^n)}$.

(ii) The authors in \cite{GHKK-JAM} and \cite{GHKK15} (see also \cite{HIK};
cf. \cite{GHK02})
gave some examples of operators
$A\in L(\Cn)$ with $m(A)>0$, for which $S_A^0(\Bn)\neq S^0(\Bn)$. On the other hand,
in view of \cite[Theorem 3.14]{GHKK-JAM}, if $A\in L(\Cn)$ with $A+A^*=2aI_n$,
for some $a>0$, where $A^*$ is the adjoint
operator of $A$, then $S_A^0(\Bn)=S^0(\Bn)$.
\end{remark}

In the last part of this section we are concerned with an approximation property of
another compact subset of
$S^0(\B^n)$ by automorphisms of $\C^n$, for $n\geq 2$. To this end, let
$${\mathcal Q}(\B^n)=\Big\{f\in H(\B^n): f(0)=0, Df(0)=I_n,
\|Df(z)-I_n\|<1,\,z\in\B^n\Big\}.$$
If $f\in {\mathcal Q}(\B^n)$, then $f\in S^0(\B^n)$ (see \cite[Lemma 2.2]{GHK06}).
If, in addition, $\|Df(z)-I_n\|\leq c$, $z\in\B^n$, for some $c\in [0,1)$,
then $f$ is quasiregular on $\B^n$ and extends to a
quasiconformal homeomorphism of $\C^n$ onto itself (see \cite[Lemma 2.2]{GHK06};
see e.g. \cite[Chapter 8]{GK}).
Note that ${\mathcal Q}(\B^n)$ is a compact subset of $H(\B^n)$.

\begin{theorem}
\label{t.qc}
If $n\geq 2$, then
${\mathcal Q}(\B^n)=\ov{{\mathcal Q}(\B^n)\cap \mathcal{A}(\B^n)}.$
\end{theorem}

\begin{proof}
Let $f\in {\mathcal Q}(\B^n)$. Since $f\in S^0(\B^n)$, there exists a sequence
$(\psi_k)_{k\in\mathbb{N}}$
in ${\rm Aut}(\Cn)$ such that $\psi_k\to f$ locally uniformly on $\B^n$, as $k\to\infty$.
We may assume that $\psi_k$ is normalized, for all $k\in\N$.
Next, let $(r_m)_{m\in\N}$ be a sequence in $(0,1)$ such that $r_m\to 1$.
Taking into account Schwarz's lemma for holomorphic mappings into complex Banach spaces,
we obtain that
$$\|Df(r_mz)-I_n\|\leq r_m,\quad z\in\B^n.$$
Since $\psi_k\to f$ locally uniformly on $\B^n$, as $k\to\infty$,
it follows that for all $m\in\N$, there exists $k_m\in\N$ such that
$$\|D\psi_{k_m}(r_mz)-I_n\|\leq \frac{1+r_m}{2},\quad z\in\B^n.$$
The sequence $(k_m)_{m\in\N}$ may be chosen to be increasing.
Now, if $\varphi_m(z)=\frac{1}{r_m}\psi_{k_m}(r_mz)$, for $z\in\Bn$, then
$$\|D\varphi_m(z)-I_n\|\leq \frac{1+r_m}{2},\quad z\in\Bn,\quad m\in\N.$$
Thus, for every $m\in\N$, $\varphi_m$ is a normalized automorphism of $\C^n$, whose
restriction to $\B^n$ belongs to the family ${\mathcal Q}(\B^n)$.
Also, $\varphi_m\to f$ locally uniformly on $\B^n$, as $m\to\infty$.
Consequently, $f\in \ov{{\mathcal Q}(\B^n)\cap
\mathcal{A}(\B^n)}$, and thus ${\mathcal Q}(\B^n)\subseteq \ov{{\mathcal Q}(\B^n)\cap
\mathcal{A}(\B^n)}$.

Clearly, the inclusion $\ov{{\mathcal Q}(\B^n)\cap
\mathcal{A}(\B^n)}\subseteq {\mathcal Q}(\B^n)$ is obvious.
This completes the proof.
 
\end{proof}

Let  $\widetilde{\mathcal Q}(\B^n)$
be the subset of ${\mathcal Q}(\B^n)$ consisting of all mappings $f$ with the power series
expansion
$f(z)=z+\sum_{k=2}^\infty A_k(z^k)$, $z\in\B^n$,
such that
\begin{equation}
\label{st-qc}
\sum_{k=2}^\infty k\|A_k\|\leq 1,
\end{equation}
where $A_k=\frac{1}{k!}D^kf(0)$ is the
$k$-th order Fr\'echet derivative of $f$ at $z=0$.
Then $\widetilde{\mathcal Q}(\B^n)\subseteq S^0(\B^n)$, since
$\widetilde{\mathcal Q}(\B^n)\subseteq {\mathcal Q}(\B^n)$.
Also, $\widetilde{\mathcal Q}(\B^n)$ is compact.

If $n=1$, then $\widetilde{\mathcal Q}(\U)\subseteq S^*(\U)$
(see e.g. \cite{GK}). It would be interesting to see if the inclusion
$\widetilde{\mathcal Q}(\B^n)\subseteq S^*(\B^n)$ remains true for
$n\geq 2$ (see \cite{GHK06}).

As in the case of the family ${\mathcal Q}(\B^n)$, we may prove the following approximation
result of the family $\widetilde{\mathcal Q}(\B^n)$, where $n\geq 2$.

\begin{theorem}
\label{t.qc2}
If $n\geq 2$, then
$\widetilde{\mathcal Q}(\B^n)=\ov{\widetilde{\mathcal Q}(\B^n)\cap \mathcal{A}(\Bn)}.$
\end{theorem}

\begin{proof}
Let $f\in \widetilde{\mathcal Q}(\B^n)$. Since $f\in S^0(\B^n)$, there exists a sequence
$(\psi_j)_{j\in\mathbb{N}}$
in ${\rm Aut}(\Cn)$ such that $\psi_j\to f$ locally uniformly on $\B^n$, as $j\to\infty$.
We may assume that $\psi_j$ is normalized, for all $j\in\N$.
Next, let $(r_m)_{m\in\N}$ be an increasing sequence in $(0,1)$ such that $r_m\to 1$.
Let $\varepsilon_m>0$ be such that
$$\frac{\varepsilon_m}{r_m}\sum_{k=2}^{\infty}k\left(\frac{2r_m}{1+r_m}\right)^k
\leq\frac{1-r_m}{2}.$$
Since $\psi_j\to f$ locally uniformly on $\B^n$, as $j\to\infty$,
it follows that for all $m\in\N$, there exists $j_m\in\N$ such that
$$\|\psi_{j_m}(z)-f(z)\|\leq \varepsilon_m,\quad \| z\|\leq \frac{1+r_m}{2}.$$
The sequence $(j_m)_{m\in\N}$ may be chosen to be increasing.
Now, if $\varphi_m(z)=\frac{1}{r_m}\psi_{j_m}(r_mz)$, for $z\in\Bn$, then
$$
\left\|\varphi_m(z)-\frac{1}{r_m}f(r_m z)\right\|
\leq
\frac{\varepsilon_m}{r_m} ,
\quad \| z\|\leq \frac{1+r_m}{2r_m},\quad m\in\N.
$$
Let $f(z)=z+\sum_{k=2}^\infty A_k(z^k)$, $z\in\B^n$,
and
$\varphi_m(z)=z+\sum_{k=2}^\infty B_k(z^k)$, $z\in\B^n$.
Then the above inequality implies that
$$\| B_k-A_kr_m^{k-1}\|\leq
\frac{\varepsilon_m}{r_m}  \left(\frac{2r_m}{1+r_m}\right)^k,
\quad k\geq 2.$$
It follows that
$$\sum_{k=2}^{\infty}k\| B_k\|
\leq
\sum_{k=2}^{\infty}k\frac{\varepsilon_m}{r_m}  \left(\frac{2r_m}{1+r_m}\right)^k
+\sum_{k=2}^{\infty}k\| A_k\|r_m^{k-1}$$
$$\leq
\frac{1-r_m}{2}+r_m=\frac{1+r_m}{2}.$$
Thus, for every $m\in\N$, $\varphi_m$ is a normalized automorphism of $\C^n$, whose
restriction to $\B^n$ satisfies (\ref{st-qc}), and thus $\varphi_m$
belongs to the family $\widetilde{\mathcal Q}(\B^n)$.
Hence, $f\in \ov{\widetilde{\mathcal Q}(\B^n)\cap
\mathcal{A}(\B^n)}$, and thus $\widetilde{\mathcal Q}(\B^n)\subseteq
\ov{\widetilde{\mathcal Q}(\B^n)\cap\mathcal{A}(\B^n)}$.

Clearly, the inclusion $\ov{\widetilde{\mathcal Q}(\B^n)\cap
\mathcal{A}(\B^n)}\subseteq \widetilde{\mathcal Q}(\B^n)$ is obvious. This completes the proof.
 
\end{proof}

Finally, we obtain an approximation result by automorphisms of $\C^n$ ($n\geq 2$)
for a subset of $K(\B^n)$. To this end, let $\widetilde{K}(\B^n)$ be the set
of all normalized mappings $f\in H(\B^n)$ with the power series expansion
$f(z)=z+\sum_{k=2}^\infty A_k(z^k)$, $z\in\B^n$, such that
\begin{equation}
\label{cv1}
\sum_{k=2}^\infty k^2\|A_k\|\leq 1.
\end{equation}
Then $\widetilde{K}(\B^n)\subseteq K(\B^n)$ (see \cite[Theorem 2.1]{RoS}). Since
the condition (\ref{cv1}) implies that $\sum_{k=2}^\infty k\|A_k\|\leq 1/2$, it follows that
$\widetilde{K}(\B^n)\subseteq \widetilde{\mathcal Q}(\B^n)$, and every mapping
$f\in \widetilde{K}(\B^n)$ is quasiregular on $\B^n$ and has a quasiconformal extension to $\C^n$
(see \cite{GHK06}). Also, $\widetilde{K}(\B^n)$ is compact.

Using arguments similar to those in the proof of Theorem \ref{t.qc2},
we obtain the following approximation result of $\widetilde{K}(\B^n)$. We omit the proof.

\begin{proposition}
\label{t.qc3}
If $n\geq 2$, then
$\widetilde{K}(\B^n)=\ov{\widetilde{K}(\B^n)\cap {\mathcal A}(\B^n)}$.
\end{proposition}

\section*{Acknowledgments}
\label{acknowledgments}
H. Hamada was partially supported by JSPS KAKENHI Grant Number
JP16K05217.
S. Schlei{\ss}inger was supported by the ERC grant "HEVO - Holomorphic Evolution Equations"
nr. 277691.

Some of the research for this paper has been carried out in August, 2016, when
Gabriela Kohr visited the Department of Mathematics of the University
of Toronto. She expresses her gratitude to the members of this department for their
hospitality during that visit.

\end{document}